\documentclass{amsart}
\usepackage{graphicx}

\newtheorem{thm}{Theorem}[section]

\newtheorem{lem}[thm]{Lemma}

\theoremstyle{definition}

\theoremstyle{remark}

\numberwithin{equation}{section}

\usepackage{amsmath,amssymb,amsfonts,euscript,graphicx}
\begin{document}

\title[ ]{On the automorphism groups of regular hyper-stars and folded hyper-stars}%
\author{S. Morteza Mirafzal }%
\address{Department of Mathematics, University of Isfahan, Isfahan 81746-73441, Iran}

\email{smortezamirafzal@yahoo.com}

\email{mirafzal@math.ui.ac.ir}%


\begin{abstract}
 The hyper-star graph $HS(n,k)$ is defined as follows : its
vertex-set is the set of $ \{0,1\} $-sequences of length $n$ with
weight $k$, where the weight of a sequence $v$ is the number of
$1^,s$ in $v$, and two vertices are adjacent if and only if one can
be obtained from the other by exchanging the first symbol with a
different symbol ( $1$ with $0$, or $0$ with $1$ ) in another
position.
  In this paper, we will find  the
automorphism groups of regular hyper-star and folded hyper-star
graphs.
 Then, we will show that, only the graphs $ HS(4,2) $ and $FHS(4,2)$ are  Cayley
 graphs. \

 \


 \


\end{abstract}

\maketitle
\section{\bf Introduction and Preliminaries}
An interconnection network can be represented as an undirected graph
where a processor is represented as a  vertex  and a communication
channel between processors as an edge between corresponding
vertices. Measures of the desirable properties for interconnection
networks include degree, connectivity, scalability, diameter, fault
tolerance, and symmetry $[1]$. For example in $[4,6]$ have been
found the symmetries of two important classes of graphs.  The main
aim of this paper is to study the symmetries of a class of graphs
that are useful in some aspects for designing some interconnection
networks. First major class of interconnection networks is the
classical n-cubes. Star graphs were introduced by $[1]$ as a
competitive model to the n-cubes. Both the n-cubes and Star graphs
have been studied and many of the properties are known and star
graphs have proven to be superior to the n-cubes. The hyper-star
graphs were introduced in $[8]$ as competitive model to both n-cubes
and star graphs. Some of the structural and topological properties
of hyper-star graphs have been studied in $[3,7]$. For all the
terminology and notation not defined here, we follow $[2,5,10]$. Let
$n>2$, the hyper-star graph $HS(n,k)$ where, $1\leq k \leq n-1 $, is
defined in $[8]$ as follows : its vertex-set is the set of $ \{0,1\}
$-sequences of length $n$ with weight $k$, where the weight of the
sequence $v$ is the number of $1^,s$ in $v$, and two vertices are
adjacent if and only if one can be obtained from the other by
exchanging the first symbol with a different symbol ( $1$ with $0$
or $0$ with $1$ ) in another position. Formally, if we denote by $
V(HS(n,k))$ and $E(HS(n,k))$ the vertex-set and edge-set of
$HS(n,k)$ respectively, then

$ V=V(HS(n,k))=\{x_1x_2\cdots x_n \mid x_i \in\{0 ,1 \}, \sum
_{j=1}^{n}x_j =k \}$

$ E = E(HS(n,k)) = \{\{u,v \} \mid u=x_1x_2 \cdots x_n , v = x_ix_2
\cdots x_{i-1}x_1x_{i+1}\cdots$ $ x_n,
 x_1 =
 x_i^c  \} $, where $x^c$
is the complement of $ x$ $ ( 0^c = 1 $   and   $ 1^c = 0 )$. It is
clear that the degree of a vertex $v$ of $HS(n,k)$ is, $n-k $ if
$1\in v $,  or is $k$ if $1\notin v $. So $HS(n,k)$ is regular if
and only if $ n = 2k $.

Let $X=\{1,2,...,n\}$ and $X_k $ be the family of subsets of $X$
with $k$ elements. Let $S(n,k)$ be the graph  with vertex-set $X_k$
and two vertices $v = \{x_1,\cdots ,x_k\}$ and $w = \{y_1,\cdots
,y_k\}$ are adjacent if and only if $\mid v \cap w \mid  = k-1   $
and, $1$ belongs to one, and only one,  of the vertices $v$ and $w$,
in other words $w$ is obtained from $v$ by replacing an element $y
\in X-v $ with $1$,  if $1\in v $, and replacing $x \in v$ by $1$
if, $1 \notin v$. Let $A$ be a subset of $X$, then the
characteristic function of $A$ is the function $ \chi_{A} :X
\longrightarrow \{0,1\}$ such that $\chi_{A}(x) =1$, if and only if
$x\in A$. Thus $ A \longmapsto \chi_{A} $ is a bijection between the
family of subsets of $X$ and the set of sequences of $ \{{0,1\}} $
of length $n$. The graphs $\Gamma_1 = (V_1,E_1)$ and $\Gamma_2 =
(V_2,E_2)$ are called isomorphic, if there is a bijection $\alpha :
V_1 \longrightarrow V_2 $   such that, $\{a,b\} \in E_1$ if and only
if $\{\alpha(a),\alpha(b)\} \in E_2$ for all $a,b \in V_1$. in such
a case the bijection $ \alpha $ is called an isomorphism. Now it is
an easy task to show that the graphs $HS(n,k)$ and $S(n,k)$ are
isomorphic, in fact the correspondence $A \longmapsto \chi_{A}$ is
an isomorphism between $S(n,k)$ and $HS(n,k)$, and for this reason,
from now on, we work with $S(n,k)$ and we denote it by $HS(n,k)$.
The following figure shows the graph $HS(6,3)$, where the set
$\{x,y,z\}$ is denoted by $xyz$.

\begin{figure}[h]
\def\emline#1#2#3#4#5#6{%
\put(#1,#2){\special{em:moveto}}%
\put(#4,#5){\special{em:lineto}}}
\def\newpic#1{}
%
%
%
\unitlength 0.4mm
\special{em:linewidth 0.5pt}
\linethickness{0.4pt}
\begin{picture}(0,0)(100,80)

%
\put(-5,29){\circle*{2}} \put(-4,8){\circle*{2}}
\put(18,-14){\circle*{2}} \put(18,51){\circle*{2}}
\put(104,73){\circle*{2}} \put(105,-35){\circle*{2}}
\put(190,51){\circle*{2}} \put(211,29){\circle*{2}}
\put(211,8){\circle*{2}} \put(190,-14){\circle*{2}}
\put(105,-22){\circle*{2}} \put(105,60){\circle*{2}}
\put(39,29){\circle*{2}} \put(41,7){\circle*{2}}
\put(83,29){\circle*{2}} \put(83,5){\circle*{2}}
\put(168,29){\circle*{2}} \put(127,3){\circle*{2}}
\put(168,5){\circle*{2}} \put(168,29){\circle*{2}}
\put(126,32){\circle*{2}}
\emline{-5}{29}{1}{-4}{8}{2}
\emline{-5}{29}{1}{18}{51}{2}
\emline{-5}{29}{1}{83}{5}{2}
\emline{-4}{8}{1}{18}{-14}{2}
\emline{-4}{8}{1}{168}{29}{2}
\emline{18}{-14}{1}{105}{-35}{2}
\emline{18}{-14}{1}{41}{7}{2}
\emline{18}{51}{1}{104}{73}{2}
\emline{18}{51}{1}{39}{29}{2}
\emline{104}{73}{1}{190}{51}{2}
\emline{104}{73}{1}{105}{60}{2}
\emline{105}{-35}{1}{190}{-14}{2}
\emline{105}{-35}{1}{105}{-22}{2}
\emline{190}{51}{1}{211}{29}{2}
\emline{190}{51}{1}{168}{29}{2}
\emline{211}{29}{1}{211}{8}{2}
\emline{211}{29}{1}{41}{7}{2}
\emline{211}{8}{1}{190}{-14}{2}
\emline{211}{8}{1}{126}{32}{2}
\emline{190}{-14}{1}{168}{5}{2}
\emline{105}{-22}{1}{83}{5}{2}
\emline{105}{-22}{1}{127}{3}{2}
\emline{105}{60}{1}{83}{29}{2}
\emline{105}{60}{1}{126}{32}{2}
\emline{39}{29}{1}{41}{7}{2}
\emline{39}{29}{1}{127}{3}{2}
\emline{83}{29}{1}{83}{5}{2}
\emline{83}{29}{1}{168}{5}{2}
\emline{168}{29}{1}{168}{5}{2}
\emline{127}{3}{1}{126}{32}{2}
\put(88,-45){\makebox(0, 0)[cc]{Fig. 1. HS(6,3) graph}}
\put(89,-22){\makebox(0, 0)[cc]{356}} \put(152,2){\makebox(0,
0)[cc]{164}} \put(108,2){\makebox(0, 0)[cc]{135}}
\put(68,1){\makebox(0, 0)[cc]{136}} \put(40,-1){\makebox(0,
0)[cc]{125}} \put(153,38){\makebox(0, 0)[cc]{642}}
\put(126,41){\makebox(0, 0)[cc]{435}} \put(71,38){\makebox(0,
0)[cc]{436}} \put(196,54){\makebox(0, 0)[cc]{142}}
\put(114,-37){\makebox(0, 0)[cc]{156}} \put(196,-15){\makebox(0,
0)[cc]{654}} \put(219,8){\makebox(0, 0)[cc]{154}}
\put(219,30){\makebox(0, 0)[cc]{524}} \put(37,39){\makebox(0,
0)[cc]{523}} \put(-2,-11){\makebox(0, 0)[cc]{256}}
\put(-20,12){\makebox(0, 0)[cc]{126}} \put(-19,37){\makebox(0,
0)[cc]{326}} \put(14,61){\makebox(0, 0)[cc]{123}}
\put(112,61){\makebox(0, 0)[cc]{143}} \put(84,81){\makebox(0,
0)[cc]{423}}

\end{picture}
\end{figure}
\

\

\

\

\

\

\

\

\

An automorphism of a graph $\Gamma $ is an isomorphism of $\Gamma $
with itself. The set of all automorphisms of $\Gamma$, with the
operation of composition of functions, is a group, called the
automorphism group of $\Gamma$ and denoted by $ Aut(\Gamma)$. A
permutation of a set is a bijection of it with itself.  The group of
all permutations of a set $V$ is denoted by $Sym(V)$, or just
$Sym(n)$ when $\mid V \mid =n $. A permutation group $G$ on $V$ is a
subgroup of $Sym(V)$. In this case we say that $G$ acts on $V$. If
$\Gamma$ is a graph with vertex-set $V$, then we can view each
automorphism as a permutation of $V$, and so $Aut(\Gamma)$ is a
permutation group. Let $G$ act on $V$, we say that $G$ is transitive
( or $G$ acts transitively on $V$ ) if there is just one orbit. This
means that given any two elements $u$ and $v$ of $V$, there is an
element $ \beta $ of  $G$ such that $\beta (u)= v $.

The graph $\Gamma$ is called vertex transitive if  $Aut(\Gamma)$
acts transitively on $V(\Gamma)$.The action of $Aut(\Gamma)$ on
$V(\Gamma)$ induces an action on $E(\Gamma)$  by the rule
$\beta\{x,y\}=\{\beta(x)$, $\beta(y)\}, \beta\in Aut(\Gamma)$, and
$\Gamma$ is called edge transitive if this action is transitive.The
graph $\Gamma$ is called symmetric, if for all vertices $u, v, x,
y,$ of $\Gamma$ such that $u$ and $v$ are adjacent, and $x$ and $y$
are adjacent, there is an automorphism $\alpha$ such that
$\alpha(u)=x,   and, \alpha(v)=y$. It is clear that a symmetric
graph is vertex transitive and edge transitive.

For $v\in V(\Gamma)$ and $G=Aut(\Gamma)$, the stabilizer subgroup
$G_v$ is the subgroup of $G$ containing all automorphisms which fix
$v$. In the vertex transitive case all stabilizer subgroups $G_v $
are conjugate in $G$, and consequently isomorphic, in this case, the
index of $G_v$ in $G$ is given by the equation,  $ | G : G_v |
=\frac{| G |}{| G_v |} =| V(\Gamma)| $. If each stabilizer $ G_v $
is the identity group, then every element of $G$, except the
identity, does not fix any vertex, and we say that $G$ acts
semiregularly on $V$. We say that $G$ acts regularly on $V$ if and
only if $G$ acts transitively and semiregularly on $ V$
 and in this case we
have $\mid V \mid = \mid G \mid$  .

Let $G$ be any abstract finite group with identity $1$, and
suppose that $\Omega$ is a set of generators of $G$, with the
properties :

(i) $x\in \Omega \Longrightarrow x^{-1} \in \Omega; (ii) 1\notin
\Omega $ ;

The Cayley graph $\Gamma=\Gamma (G, \Omega )$ is the ( simple )
graph whose vertex-set and edge-set defined as follows :

$V(\Gamma) = G  ;   E(\Gamma)=\{\{g,h\}\mid g^{-1}h\in \Omega \}$.
It can be shown that a  connected graph $\Gamma$ is a cayley graph
if and only if $Aut(\Gamma)$ contains a subgroup $H$,  such that $H$
acts regularly on $V(\Gamma)$ $[2,5]$.

The group $G$ is called a semidirect product of $ N $ by $Q$,
denoted by $ G = N \rtimes Q $,
 if $G$ contains subgroups $ N $ and $ Q $ such that, (i)$
N \unlhd G $ ($N$ is a normal subgroup of $G$ ); (ii) $ NQ = G $;
(iii) $N \cap Q =1 $.

\section{\bf Main results}\
In the remaining of this section we assume that $ k $ is a fixed
natural number, but arbitrarily chosen and $ k>2 $ and $
X=\{1,2,...,2k \}$.

\begin{lem}\label{1}The graph $HS(2k,k)$ is a vertex transitive
graph.

\end{lem}

\begin{proof}\ In $[8]$ it is proved that $ HS(2k,k)$ is a vertex
transitive graph and in $[3]$ it is proved that this graph is edge
transitive, but for the sake of consistency and,  since our proof is
independent of those and  we need our proof in the sequel, we bring
a proof. Let $V=V(HS(2k,k))$. The graph $HS(2k,k)$ is a regular
bipartite graph of valency (regularity $k$ ), in fact if $P_1 = \{v
\in V \mid 1\in v \}$ and $P_2 = \{w \in V \mid 1\notin w \}$ then,
$\{P_1, P_2 \}$ is a partition of $V$ and every edge of $HS(2k,k)$
has a vertex in $P_1$ and a vertex in $P_2$ and $\mid P_1 \mid =
\mid P_2 \mid $. Let $ \alpha $ be a permutation of $ Sym(X) $ such
that $ \alpha $ fixes the element $1$. $ \alpha $
 induces a permutation $\tilde{\alpha}$ on $V$ by the rule
  $\tilde{\alpha}(\{x_1,x_2,\cdots,x_k \}) =
 \{\alpha(x_1), \alpha(x_2),\cdots, \alpha(x_k) \} $.
  We have $ \mid v\cap w \mid$  = $ \mid \tilde{\alpha}(v) \cap \tilde{\alpha}(w) \mid
  $ and $1$ is in one, and only one of the vertices of an edge,
  thus   $\tilde{\alpha}$ is an automorphism of the graph
  $HS(2k,k)$. Note that if $ v \in P_1 $, then $ \tilde{\alpha}(v) \in P_1
  $,   thus   $ \tilde{\alpha}(P_1) =P_1 $ and $ \tilde{\alpha}(P_2) =P_2
  $. For any  vertex $v$ in $V$,  let $v^c$ be the complement of the
  set $v$ in $X$. We define the mapping  $\theta :V \longrightarrow V $ by the
  rule,  $ \theta(v) =v^c $, for every $v$ in $V$. In fact $ \theta
  $ is an automorphism of $HS(2k,k)$. Note that for any $\alpha$
  in $ Sym(X)$ that fixes $1$ , $\tilde{\alpha} \neq \theta $.
 Now, let $v, w \in V $. Suppose  $v, w \in P_1$ and $| v \cap w | = t $.
 Let $v =\{1, x_2,... , x_t,y_1,..., y_{k-t} \} $ and
  $ w =\{1, x_2,... , x_t,z_1,..., z_{k-t} \}$. We define the
  permutation $\pi \in Sym(X)$ by the rule; $ \pi(1) = 1, \pi (x_i) = x_i, \pi (y_j)= z_j,
  $ and $ \pi(u) = u, u \in X-( v \cup w ) $. Thus, $\tilde{\pi}$
  is an automorphism of $HS(2k,k)$ and $ \tilde{\pi}(v) = w $. If
  $ v,w \in P_2 $ then, $ \theta(v),\theta(w) \in P_1 $, therefore
  there is an automorphism $\tilde{\pi}$ in $Aut(HS(2k,k))$ such
  that $\tilde{\pi}(\theta(v)) = \theta (w)$, thus $ (\theta^{-1} \tilde{\pi} \theta)(v)
  =w $. Now, let $ v\in P_1 $ and $ w \in P_2 $, thus $ \theta(w) \in P_1
  $ and there is an automorphism $\tilde{\pi} \in Aut(HS(2k,k))$ such that
  $ \tilde{\pi} (v) = \theta(w) $,  then we have $ \theta^{-1}\tilde{\pi}(v) = w
  $.

\end{proof}

 For a graph $\Gamma$ and $v\in V(\Gamma)$, let $N(v)$
be the set of vertices $w$ of $\Gamma$ such that $w$ is adjacent to
$v$. If $ G= Aut( \Gamma)$, then $G_v$ acts on $ N(v)$,  if we
restrict the domains of the permutations $g \in G_v $ to $ N(v)$ .
It is an easy task to show that a vertex transitive graph $\Gamma$
is symmetric, if and only if, $G_v$ acts transitively on the set
$N(v)$ for any $v\in V(\Gamma)$. In the sequel $ \theta $ is the
automorphism of $HS(2k,k)$  which is defined in Lemma 2.1.

\begin{thm}\  The graph $HS(2k,k)$ is a  symmetric
graph.

\end{thm}

\begin{proof}\ Let $ \Gamma = HS(2k,k)$ and  $ G = Aut(HS(2k,k))$. Since $ \Gamma $ is a
vertex transitive graph, it is enough to show that $ G_v $ acts
transitively on $ N(v)$ for any $ v \in V =V(\Gamma)$. Let $ v \in
P_1 $, $ v= \{1, x_2,...,x_k \}$, thus $ N(v) = \{\{y_i, x_2,...,x_k
\} \mid  1 \leq i \leq k\}$, where $ X = \{ 1, x_2,...,
x_k,y_1,...,y_k\}$. If $w_i,w_j \in N(v)$, $w_i = \{y_i, x_2,...,x_k
\}$, $ w_j = \{y_j,x_2,...,x_k \}  $, then the transposition $\tau
=(y_i y_j) \in Sym(X)$ is such that $ \tilde{\tau}$ is in $G_v$ and
$ \tilde{\tau}(w_i) = w_j $.
 Now let
 $ v\in P_2 $ and  $ u,w \in  N(v)$, thus $ \theta (v) = v^c \in P_1 $
 and $ \theta (u), \theta(w) \in N ( \theta(v)) = N(v^c)$.
 Therefore there is an automorphism $ \pi \in G_{v^{c}}$ such that
  $ \pi( \theta (u)) = \theta(v)$. Thus, $ (\theta^{-1} \pi \theta)(u) = w $
   and since $ \pi ( \theta(v)) = \theta(v) $, we have $(\theta^{-1} \pi \theta)(v) = v
   $.
\end{proof}

   \ \ \    Suppose $ \Gamma $ is a graph
 and $ G = Aut( \Gamma )$. For a
vertex $ v $ of $ \Gamma $,  let $ L_v $ be the set of all elements
$ g $ of $G_v$ such that $ g $ fixes  each element of $ N(v)$. Let $
L_{v,w} = L_v \cap L_w $.

\begin{lem} \ Let $ \Gamma $ be a graph such that every vertex of
it is of degree greater than one and $ G = Aut( \Gamma )$. If $ v $
be a vertex of $ \Gamma $ of degree $ b $, and $ w $ be an element
of $ N(v)$ with minimum  degree $ m $, then, $ \mid G_v \mid \leq
b!(m-1)!|L_{vw}|$.

\end{lem}

\begin{proof}\ Let $ Y = N(v) $ and $ \Phi : G_v \longrightarrow
Sym(Y)$ be defined by the rule, $ \Phi(g) = g_{\mid Y}$ for any
element $g$ in $G_v$, where $ g_{\mid Y}$ is the restriction of $ g
$ to $Y$. In fact $ \Phi $ is a  group homomorphism and $ ker(\Phi)
= L_v $,  thus $ G_v/L_v$ is isomorphic with a subgroup of $ Sym(Y)
$. Since, $\mid Y \mid = deg(v) = b$, therefore $ \mid G_v \mid /
\mid L_v \mid \leq b! $. \

  Now,  $ \mid G_v \mid \leq (b!)\mid L_v
\mid $. If $w$ is an element of $ N(v)$ of degree $l$ and $ g \in
L_v $, then $g$ fixes $ v\in N(w)$. Let $ Z = N(w) - \{v \}$ and $
\Psi : L_v \longrightarrow Sym(Z) $ be defined by $ \Psi (h) =
h_{\mid Z}$, for any element $h$ in $L_v $. Then the kernel of the
homomorphism $\Psi$ is $ L_{vw} $ and since $ \mid Z \mid = l-1 $,
thus $ \mid L_v \mid \leq (l-1)! \mid L_{v,w} \mid $. Now, we have $
\mid G_v \mid \leq b!(l-1)! \mid L_{v,w} \mid $. If $ w $ be an
element in $ N(v) $ of minimum degree $m$, then the result follows.

\end{proof}

\ \ \ \ From the previous Lemma it  follows that, if $\Gamma $ is a
regular graph of degree $m$,  then for every edge $ \{v,w \} $ of $
\Gamma $ we have $ \mid G_v \mid \leq m!(m-1)! \mid L_{vw}\mid $.

\begin{thm}\ The automorphism group of $ HS(2k,k) $ is a
semidirect product of $ N $ by $ Q $, where $ N $ is isomorphic to $
Sym(2k-1) $ and $ Q $ is isomorphic to $ Z_2 $, the cyclic group of
order $2$.

\end{thm}

\begin{proof}\ If $ H $ be the subgroup of $ Sym(X)$ that
contains permutations which fix the element $ 1 $, then $ H $ is
isomorphic with $ Sym(2k-1) $.  Then $ f : H \longrightarrow
Aut(HS(2k,k)) = G $, defined by $f(\alpha) = \tilde{\alpha}$, ( $
\tilde{\alpha} $ is defined in Lemma 2.1) is an injection. In fact,
if $ \alpha \neq 1 $ be in $ Sym(X) $ and $ \alpha (1) = 1 $, then
there is an $ x \in X $ such that $ \alpha (x) \neq x $. Now, let $T
$ be a $k$-subset of $X$ such that $ x \in T $ and $ \alpha (x)
\notin T $. Then $ \tilde{\alpha}(T) \neq T $ and hence
$\tilde{\alpha} \neq 1 $. It follows that the kernel of the
homomorphism $f$ is the identity group.  Therefore, the subgroup $
f(H) = N = \{ \tilde{ \alpha} \mid \alpha \in H \} $ is of order $(
2k-1)!$. If $ Q $ be the cyclic subgroup of $ G $ generated by $
\theta $ ($ \theta $  is defined in Lemma $2.1 $), then $ \mid Q
\mid = 2 $. Since, $ \theta \notin N $, so $N \cap Q = 1 $, thus for
the set $ NQ \subseteq G $ we have $ \mid NQ \mid = \frac{ \mid N
\mid \mid Q \mid}{1} = (2k-1)!(2) $,  so we have $ \mid G \mid \geq
(2k-1)!(2)$. If we show that $ \mid G \mid \leq (2k-1)!(2)$, then we
must have $ G = NQ $ and since the index of $ N $ in $ NQ = G $ is $
2 $,  then $ N $ is a normal subgroup of $ G $ and the theorem will
be proved. In the first step of the remaining proof, we assert  that
every 3-path in the graph $ \Gamma = HS(2k,k)$ determines a unique
6-cycle in this graph. Let $ P : v_1v_2v_3v_4 $ be a 3-path in
$\Gamma $. The path $ P $ has a form such as, $ v_1 = \{y_1,x_2,x_3,
...x_k \} v_2 =\{ 1,x_2,x_3,...,x_k \} v_3 = \{ y_2, x_2,x_3,...,x_k
\} v_4 = \{ y_2,1,x_3,...,x_k \}$. If  $ C $ be a 6-cycle of $
\Gamma $ that contains $ P $, then $C $ has two adjacent vertices $
v_5 $ and $ v_6 $ such that $ v_5 $ is adjacent to $ v_4 $ and $ v_6
$ is adjacent to $ v_1 $. Thus $ v_5 $ has a form such as  $ v_5 =
\{y_2, s,x_3,...,x_k \}  $ where,  $s\in \{y_1, y_3,...,y_k, x_1\}$
and $ v_6 $ has a form such as $ v_6 = \{
y_1,x_2,...,x_{i-1},1,x_{i+1},...,x_k\} $. Since $ v_5 $ and $ v_6 $
are adjacent we must have $ v_5 = \{ y_2, y_1,x_3,...,x_k\} $ and $
v_6 = \{y_1, 1,x_3,...,x_k \} $. Now the assertion is proved. In the
second step we show that if $\{ v,w\}$ be an edge of $ \Gamma $,
then $ L_{v,w} = 1 $. Let $ g \in L_{v,w} $ and $ x $ be a vertex of
$ \Gamma $ of distance $ 2 $ from $ v $. If $ x $ is adjacent to $ w
$, then $ g( x ) = x $. Let $ x $ is not adjacent to $ w $, so there
is a vertex $ y $ adjacent to $ v $
 such that $ vyx $ is a 2-path of $ \Gamma $.
 If $ C : xyvwtu $ be the unique 6-cycle
that contains the 3-path $ xyvw $,  then $ g(C) $ is the 6-cycle $
g(x)yvwtg(u) $, so $C$ and $g(C)$ contain the 3-path $ yvwt $, thus
$ g(C) = C $.   Therefore $ g_{\mid V(C)} $ is an automorphism of
6-cycle $ C $ that fixes the 2-path $ wvy $, thus $ g $ fixes all
vertices of this cycle and we have $ g(x) = x $. Now,  since the graph
$ \Gamma $ is connected , it follows that $ g $ fixes all the
vertices of $ \Gamma $,  so  $ g = 1 $ and $ L_{v,w} = 1 $.

The graph $ \Gamma $ is vertex transitive, thus for a vertex $
v\in V = V(\Gamma)$ we have;
$$ \mid G \mid =
 \mid V \mid \mid G_v \mid \leq {2k\choose k } (k!)(k-1)! = \frac {2k! } {k!k! } k!(k-1)!
 = \frac{ 2k!}{k} = (2k-1)!2 $$

\end{proof}

\ \ \ \ \  {\bf Remark:} As we can see in the proof of Theorem $ 2.4 $,
the graph $ HS(2k,k)$ has 6-cycles and since this graph is bipartite
, hence it has no 3-cycles and no 5-cycles. It is easy to show that
this graph has no 4-cycles, so the girth of this graph is 6. \

\

  \section{\bf Folded hyper-star graphs}

   The folded hyper star-graph $ FHS(2k,k) $ is the graph which
its vertex-set is identical to the vertex-set of hyper-star graph $
HS(2k,k)$, and with  edge-set $ E_2 = E_1 \cup \{  \{ v,v^c\} \mid
v\in V_1\}$, where $ E_1 $ and $ V_1 $ are the edge-set and
vertex-set of  $ HS(2k,k)$ respectively. It is clear that this graph
is a regular bipartite graph of degree $ k+1 $. It is an easy task
to show that the diameter of $ FHS(2k,k) $  is $ k $, whereas the
diameter of $ HS(2k,k)$ is $ 2k-1 $ $[8]$. We will show that this
graph is also vertex transitive, thus its edge connectivity is
maximum, say $k+1$ $[5,11]$ . Let $ v $ be a vertex of $ FHS(2k,k) $.
We can suppose that $ v = \{1,x_2,...,x_k \} $, then $ N (v) =  \{
\{y_i,x_2,...,x_k \}, 1 \leq i \leq k  \} \cup \{\{y_1,...,y_k \}\}
$, where $ X = \{1,x_2,...,x_k,y_1,...,y_k \}$. Then for every $w
\in  N(v) $ and $ w \neq v^c ,  w^c $ is the unique vertex that is
in  $ N(v^c) $ and adjacent to $ w $. Thus, if $ \{\ v,w\} $ be an
edge of this graph and $ v \neq w^c $, then the 4-cycle $ vww^cv^c $
is the unique 4-cycle that contains this edge, whereas if $ w = v^c
$ , then any 4-cycle $ vv^cu^cu $, where $  u $ is adjacent to $ v
$, contains this edge. Let $ uwv v^c $ be a 3-path in $ FHS(2k,k) $
and $ u \neq w^c $ , then by a similar way that we have seen in the
proof of Theorem 2.4 , we can  show that the 6-cycle $ uwvv^cw^cu^c
$ is the unique 6-cycle that contains this 3-path. It is clear that
the girth of this graph is 4. The following figure shows $ HS(4,2) $
graph and $ FHS(4,2)$ graph.

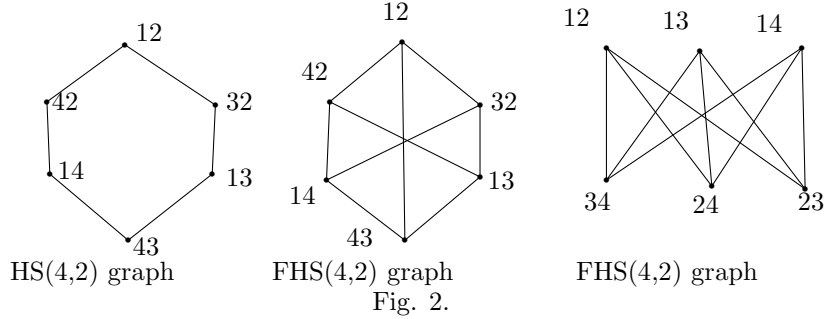
\begin{figure}[h]
\def\emline#1#2#3#4#5#6{%
\put(#1,#2){\special{em:moveto}}%
\put(#4,#5){\special{em:lineto}}}
\def\newpic#1{}
%
%
%
\unitlength 0.4mm
\special{em:linewidth 0.4pt}
\linethickness{0.4pt}
\begin{picture}(0,0)(100,80)
%
\put(7,62){\circle*{2}} \put(-19,43){\circle*{2}}
\put(-18,19){\circle*{2}} \put(8,-3){\circle*{2}}
\put(37,42){\circle*{2}} \put(36,19){\circle*{2}}
\put(125,42){\circle*{2}} \put(74,17){\circle*{2}}
\put(100,-3){\circle*{2}} \put(125,18){\circle*{2}}
\put(99,63){\circle*{2}} \put(75,43){\circle*{2}}
\put(232,61){\circle*{2}} \put(167,17){\circle*{2}}
\put(202,15){\circle*{2}} \put(167,61){\circle*{2}}
\put(198,60){\circle*{2}} \put(233,14){\circle*{2}}
\emline{7}{62}{1}{-19}{43}{2}
\emline{7}{62}{1}{37}{42}{2}
\emline{-19}{43}{1}{-18}{19}{2}
\emline{-18}{19}{1}{8}{-3}{2}
\emline{8}{-3}{1}{36}{19}{2}
\emline{37}{42}{1}{36}{19}{2}
\emline{125}{42}{1}{74}{17}{2}
\emline{125}{42}{1}{125}{18}{2}
\emline{125}{42}{1}{99}{63}{2}
\emline{74}{17}{1}{100}{-3}{2}
\emline{74}{17}{1}{75}{43}{2}
\emline{100}{-3}{1}{125}{18}{2}
\emline{100}{-3}{1}{99}{63}{2}
\emline{125}{18}{1}{75}{43}{2}
\emline{99}{63}{1}{75}{43}{2}
\emline{232}{61}{1}{167}{17}{2}
\emline{232}{61}{1}{202}{15}{2}
\emline{232}{61}{1}{233}{14}{2}
\emline{167}{17}{1}{167}{61}{2}
\emline{167}{17}{1}{198}{60}{2}
\emline{202}{15}{1}{167}{61}{2}
\emline{202}{15}{1}{198}{60}{2}
\emline{167}{61}{1}{233}{14}{2}
\emline{198}{60}{1}{233}{14}{2}
\put(104,-24){\makebox(0, 0)[cc]{Fig. 2. }}
\put(187,-14){\makebox(0, 0)[cc]{FHS(4,2)  graph}}
\put(86,-14){\makebox(0, 0)[cc]{FHS(4,2)  graph}}
\put(-4,-14){\makebox(0, 0)[cc]{HS(4,2)  graph}}
\put(85,-2){\makebox(0, 0)[cc]{43}} \put(235,10){\makebox(0,
0)[cc]{23}} \put(200,9){\makebox(0, 0)[cc]{24}}
\put(164,10){\makebox(0, 0)[cc]{34}} \put(221,69){\makebox(0,
0)[cc]{14}} \put(190,70){\makebox(0, 0)[cc]{13}}
\put(157,71){\makebox(0, 0)[cc]{12}} \put(66,12){\makebox(0,
0)[cc]{14}} \put(132,17){\makebox(0, 0)[cc]{13}}
\put(133,43){\makebox(0, 0)[cc]{32}} \put(70,53){\makebox(0,
0)[cc]{42}} \put(97,73){\makebox(0, 0)[cc]{12}}
\put(14,-5){\makebox(0, 0)[cc]{43}} \put(45,18){\makebox(0,
0)[cc]{13}} \put(-13,43){\makebox(0, 0)[cc]{42}}
\put(-11,20){\makebox(0, 0)[cc]{14}} \put(45,43){\makebox(0,
0)[cc]{32}} \put(15,66){\makebox(0, 0)[cc]{12}}
\end{picture}
\end{figure}

\

\

\

\

\

\

\

\begin{thm}\ The automorphism group of folded hyper-star graph
$ FHS(2k,k)$ is identical to the automorphism group of hyper-star
graph $ HS(2k,k)$.

\end{thm}

\begin{proof}\ Let $ \Gamma_1 = HS(2k,k) $ ,  $ \Gamma_2 =
FHS(2k,k)$ and $ H = NQ $ be the set which is defined in the proof
of Theorem 2.4 . Let $ \{v,w \} = e $ be an edge of $ \Gamma_2 $ and
$h \in NQ $. If $ e $ be an edge of $ \Gamma_1 $, then $ h(e) $ is
an edge of $ \Gamma_2 $. If $ e $ is not an edge of $ \Gamma_1 $,
then $ w = v^c $. Let $ h = nq $ ,  $ n \in N $, $ q \in  Q $, then
we have $ h(e) = \{ h(v),h(v^c)\} = \{ nq(v),nq(v^c)\} =
\{n(v),n(v^c \}$, now since,  $ \mid n(v) \cap n(v^c) \mid = \mid  v
\cap v^c \mid = 0 $, then  $ h(e) $ is an edge of the graph $
\Gamma_2 $. It follows that $ H = NQ \leq  Aut( \Gamma_2) $.  Then $
\mid Aut( \Gamma_2) \mid \geq \mid NQ \mid = (2k-1)!2  $. Let $ G =
Aut(\Gamma_2)$. If $ \{v,w \} $ be an edge of $ \Gamma_2 $ such that
$ w \neq v^c $, then we will  show that $ L_{v,w} = 1 $. Let $ g \in
L_{v,w}$. Let $ u $ be a vertex of $ \Gamma_2 $ of distance $ 2 $
from the vertex $ v $. Then there is a vertex $ t $ such that $ utv
$ is a 2-path in the graph $ \Gamma_2 $. If $ t = v^c $, then the
4-cycle,  $ C: uv^cvu^c $ is the unique 4-cycle that contains the
2-path $ v^cvu^c $. On the other hand,  the 4-cycle $ g(C) =
g(u)v^cvu^c $ also contains this 2-path, hence $ g(u) = u $. Suppose
that $ t \neq v^c,u^c $, then the path $ utvw $ is a 3-path in the
subgraph $ \Gamma_1 = HS(2k,k) $, so there is a unique 6-cycle $ C :
utvwrs $ in $ \Gamma_1 $ that contains this 3-path. $ C $ also is
the unique 6-cycle in $ \Gamma_2 $ that contains the  3-path $ utvw
$. On the other hand, $ g(C) = g(u)g(t)g(v)g(w)g(r)g(s) =
g(u)tvwrg(s)$, thus $ g(C)$ and $ C $ are 6-cycles that contains the
3-path $ tvwr $, hence $ g(C) = C $ and $ g_{\mid V(C)}$, the
restriction of $ g $ to $ V(C) $,  is an automorphism of the cycle $
C $ that fixes the vertices $ t,v, w,r,$ therefore $ g(u) = u $. If
$ u = t^c $, then $ C : vtt^cv^c $ is the unique 4-cycle that
contains the 2-path $ v^cvt $ and $ g(C) : vtg(t^c)v^c $ also
contains this 2-path, so $ g(C) = C $, then $ g(u) = u $.

Since the graph $ \Gamma_2 $ is a connected graph, thus we can
conclude that $ g(u) = u $ for any vertex  $ u $ of $ \Gamma_2 $,
then $ L_{v,w} = 1 $. \

 Let  $ v $
be a vertex of $ \Gamma_2 $, since this graph is a regular graph of
degree $ k+1 $, then from Lemma 2.3, it follows that  $ \mid G_v
\mid \leq (k+1)!k! $. Now,  We show that in fact, $ \mid G_v \mid
\leq (k-1)!k! $. Let $ v $ be a vertex of $ \Gamma_2 $, $ w \in N(v)
$ and $ w \neq v^c $. If $ g \in L_v $, then $ g $ fixes $ w $, so $
g $ induces a permutation on $ N(w) $. Since, the 4-cycles $ C :
wvv^cw^c $ and $ g(C) = wvv^cg(w^c) $, are identical, then $ g(w^c)
= w^c $. Therefore $ g $ fixes two elements $ v $ and $ w^c $ of $
N(w) $, hence $ L_v / L_{vw} \leq
 Sym(k+1-2) $, thus $ \mid L_v \mid \leq (k-1)! $. Now,  let
 $ h \in G_v $, then $ h $ induces a permutation on $ N(v)
 $, so $ h(v^c) = w $ is in  $ N(v) $. Let $ B = N(v) \cup N(v^c) - \{ v,v^c\}
 $ and $ S[B] = T $ be the subgraph induced by $ B $. It is clear
 that $ T $ is isomorphic to $ h(T) $, where $ h(T )$ is the
 subgraph induced by the set $ D = h(B) = N(v) \cup N(w)-\{ v,w\}
 $. We assert that if $ w \neq v^c $, then the subgraph induced by $ D $ has not
 any edge,
  whereas the subgraph induced by $ B $ has
 $ k $ edges. Suppose $ x,y \in D $ and $ x \in  N(v) $ and $ y \in N(w) $. We can assume that $ v = \{ 1,x_2,...,x_k \} $
  and $ w = \{ y_i,x_2,...,x_k\}$, then $ x = \{ y_j,x_2,...,x_k \}$ and $ y = \{ y_i,x_2,...x_{l-1},1,x_{l+1},...,x_k\}$,
  where $ i \neq j $. Now, it is clear that $ \{x,y \} $ is not an edge of $ \Gamma_2 $.
    Hence, $ h(v^c) = w = v^c $. Now if $ Y = N(v)- \{
 v^c\}$, then $ h_{\mid Y} \in Sym(Y) $, so $ G_v / L_v \leq
 Sym(k)$, therefore $ \mid G_v \mid \leq \mid L_v \mid (k!) \leq (k-1)!(k!)
 $. Since The graph $ \Gamma_2 $ is a vertex transitive graph,
 thus;   $$ \mid Aut(\Gamma_2) \mid =  \mid G \mid = \mid V( \Gamma_2)\mid \mid G_v \mid \leq
 \frac{2k!}{k!k!}(k!)(k-1)! = (2k-1)!2$$  Now, we have $ Aut(FHS(2k,k)) = H = NQ =
 Aut(HS(2k,k))$.

\end{proof} \
\ \ \ \    If $ k = 2$, then $ HS(2k,k) $ is isomorphic to $ C_6
 $, the cycle on 6 vertices, hence $ Aut(HS(4,2))$ is $ D_{12} $,
 the dihedral group of order 12. If $ m $ be an odd number, then $ D_{4m} = D_{2m} \rtimes Z_2 $.  Therefore $ D_{12} = D_6 \rtimes
 Z_2$, but $ D_6 \cong Sym(3) $
 , hence Theorem 2.4 is also true for $ k = 2 $. But $ FHS(4,2)$
 is isomorphic to $ K_{3,3} $, the complete bipartite graph of
 degree 3, and $ Aut(K_{3,3}) $ is a group of order $72$  $
 [2]$,
 thus Theorem 3.1 is not true for $ k = 2 $. \

 \

  \ \ \ \ The group $ G $ acting on a set $ \Omega $ induces a natural
  action on the set $ \Omega^{\{m\}}$, the set of $m$-element
  subsets of $ \Omega $, by the rule $ A^g = {\{ a_1,...,a_m\}}^g = \{
  g(a_1),...,g(a_m)\}$, where $ A \subseteq \Omega $ and $ g \in G
  $. The group $ G $ is called $m$-homogenous, if its action on $ \Omega^{\{m\}}$
   is transitive. We need the following fact.
    \

   \

   \ \ \ {\bf FACT} $[9]$. Let $ G $ be a group acting on a set $ \Omega $,
 and $ | \Omega | = n \geq 2m$, $ m\geq2$. If $ G $ is $m$-homogenous,
 then it is also $(m-1)$-homogenous.
 \begin{thm}\  Let $ k \geq 3 $. If  $ \Gamma \in \{ HS(2k,k), FHS(2k,k)\} $, then
 $ \Gamma $ is not a Cayley graph.

 \end{thm}

 \begin{proof}\ We know that $ Aut(HS(2k,k)) =  Aut(FHS(2k,k))$,
 so if $ R $ is a subgroup of $ Aut(HS(2k,k))$, then $ R $ acts regularly on $ V(HS(2k,k)) $
 if and only if $ R $ acts regularly on $ V(FHS(2k,k))
 $. Hence, it is enough to prove the theorem for $ HS(2k,k)$. Suppose the contrary, that
 $ HS(2k,k)$ is a Cayley graph, then $ Aut(HS(2k,k))$  has a subgroup $ R
 $ that
   acts regularly on $ V(HS(2k,k))
 $, then $ \mid R \mid = {2k \choose k} = \frac{2k!}{k!k!}$. If $ r
$ is an element of $ R $, then $ r =  {\tilde \sigma}\theta^i $,
where $ {\tilde \sigma} $ and $ \theta $ are defined in the proof of
Theorem 2.4 and $ i \in \{ 0,1\}$. Let $ M_1 =  \{\tilde \sigma \mid
\tilde \sigma \in R \}$, then $ M_1 $ is a subgroup of $ R $. Since
$ R $ acts on $ V(HS(2k,k))$ transitively, so $ R  $ contains an
element of the form $ \tilde { \sigma } \theta $.
 Now, if  $ M_2$ =$
\{
 \tilde {\alpha}  \theta | \tilde {\alpha}  \theta \in R \}$,
  then $ M_2 \tilde { \sigma } \theta  \subseteq M_1 $,
 because $ \tilde{\alpha} \theta \tilde{\sigma} \theta  $ = $  \tilde{\alpha} \tilde{\gamma
 }$.
 Then, $ | M_2| \leq |M_1| $. Since $ M_1 \tilde { \sigma } \theta \subseteq M_2
 $, then $ |M_1| \leq |M_2|$, so $ |M_1| = |M_2| = (1/2)R $.
   If $ M = \{\sigma \mid
\tilde \sigma \in M_1 \}$, then $ \mid M_1 \mid = \mid M \mid $ and
$ M $ is a subgroup of $ Sym(X)$ and every element of $ M $ fixes
the element $1$, where $ X = \{ 1,2,...,2k\}$. In fact $ M $ acts on
$ Y = \{2,...,2k \} $ and is $(k-1)$-homogenous on this set. Since $
2(k-1)\leq 2k-1 $, then $ M $ is $(k-2)$-homogenous on $ Y $. Hence
we must have, $ {{2k-1}\choose{k-2}} \mid \  \mid M \mid = (1/2){
{2k}\choose{k}}$, therefore $ 2(2k-1!)k!k! \mid (2k!)(k-2)!(k+1)!$,
hence $ k(k-1) \mid k(k+1)$, so $ k-1 \mid k+1 $, thus we must have
$ k \in \{1,2,3\} $. If $ k = 3 $, then $ \mid M \mid $  =$ (1/2) {6
\choose 3} $ = 10. Since $ 2 \mid {| M \mid} $, then there is an
element $ \sigma $ in $ M $ such that the order of $ \sigma $ is 2.
Note that $ \sigma $ is an element of $ Sym(6)$ that fixes $1$. If
we write $ \sigma $ in the form of a product of disjoint cycles,
then $ \sigma = (rs)$ or $ \sigma = (rs)(tu)$, where $ r,s,t,u \in
\{ 2,3,...,6\}$. In each of these cases, for $ \tilde {\sigma} \in R
$ and vertex $ v = \{ 1,r,s \}$ of $HS(6,3)$ we have $ \tilde
{\sigma} (v) = \{\sigma(1),\sigma(r),\sigma(s) \} = \{ 1,r,s\} = v
$.  Thus $ R $ can not be a regular subgroup of $ Aut(HS(6,3))$
which contradicts the assumption.

 \end{proof}
 \ \ \ If $ k = 2 $, then $ HS(4,2)$ is $ C_6 $, the cycle on 6
 vertices, which is the Cayley graph $ \Gamma = \Gamma (Z_6,
 \Omega)$, where $ Z_6 $ is the cyclic group of order 6 and $ \Omega = \{
 1,-1\}$. The graph $ FHS(4,2)$ is $ K_{3,3}$, the complete
 bipartite graph of degree 3, which is the Cayley graph $ \Gamma = \Gamma(Sym(3),
 \Omega)$, where $ \Omega = \{ (12),(23),(13)\}$ $ [2]$. \

 \

 {\bf ACKNOWLEDGMENT}    \

 The author is grateful to professor Alireza Abdollahi and professor A. Mohammadi Hassanabadi for their  helpful comments and thanks the Center of Excellence for Mathematics, University of Isfahan.

\end{document}